\documentclass[12pt,oneside]{article}
\usepackage{amsmath,amssymb,amsfonts,amsthm}
\usepackage{color}
\usepackage[all]{xy}
\usepackage{graphicx}
\usepackage{color}
\usepackage{makeidx}
\usepackage{mathtools}
\usepackage{multirow}
\usepackage{lscape}
\usepackage{rotating}
\newcommand{\T}{{\cal T}}
\newcommand{\C}{{\cal C}}

\newcommand{\Real}{\mathbb R}

\newcommand{\To}{\longrightarrow}

\newcommand{\p}{\pi^{-1}(TM)}

\newcommand {\cppp}{\mathfrak{X}(TM)}

\def\pa{\partial}
\def\paa{\dot{\partial}}

\newtheorem{thm}{Theorem}[section]

\newtheorem{lem}[thm]{Lemma}
\newtheorem{prop}[thm]{Proposition}
\newtheorem{defn}[thm]{Definition}

\numberwithin{equation}{section}

\newcommand\undersym[2]{\raisebox{-7pt}{\tiny$#2$}{\kern-8pt}\mbox{$#1$}}
\newcommand\undersymm[2]{\raisebox{-8pt}{\tiny$#2$}{\kern-15pt}\mbox{$#1$}}
\newcommand\overast[1]{\raisebox{9pt}{\small$\ast$}{\kern-9pt}\mbox{$#1$}}
\newcommand\overlind[1]{\raisebox{10pt}{\small$\overline{{\hspace{2pt}}\star}$}{\kern-7.5pt}\mbox{$#1$}}
\newcommand\overlinc[1]{\raisebox{10pt}{\tiny$\overline{{\hspace{2pt}}\circ}$}{\kern-7.5pt}\mbox{$#1$}}
\newcommand\overlina[1]{\raisebox{10pt}{\small$\overline{{\hspace{1pt}}\ast}$}{\kern-7.5pt}\mbox{$#1$}}
\newcommand\overcirc[1]{\raisebox{10pt}{\tiny{$\circ$}}{\kern-7.5pt}\mbox{$#1$}}
\newcommand\overdiamond[1]{\raisebox{10pt}{\small$\star$}{\kern-7.5pt}\mbox{$#1$}}

\newcommand\tovercirc[1]{\raisebox{5pt}{\tiny{$\circ$}}{\kern-5.5pt}\mbox{$#1$}}
\newcommand\toverdiamond[1]{\raisebox{5pt}{\tiny$\star$}{\kern-5.5pt}\mbox{$#1$}}
\newcommand\toverast[1]{\raisebox{5pt}{\tiny$\ast$}{\kern-5pt}\mbox{$#1$}}

\makeindex

\begin{document}
\title{\bf{The existence and uniqueness of    Hashiguchi   connection in KG-approach } }
\author{\bf{S. G. Elgendi$^{1}$  and A. Soleiman$^{2}$}}
\date{}
\maketitle                     
\vspace{-1.0cm}
\begin{center}
 $^{1}$ Department of Mathematics, Faculty of Science, Islamic University of Madinah,  Madinah, Kingdom of Saudia Arabia
\vspace{0.2cm}
\\
{$^{2}$ Department of Mathematics, Faculty of Science,
Al Jouf University, Skaka,  Kingdom of Saudia Arabia
}
\end{center}
\vspace{-0.5cm}
\begin{center}
E-mails: selgendi@iu.edu.sa, salahelgendi@yahoo.com\\
{\hspace{1.8cm}}asoliman@ju.edu.sa, amrsoleiman@yahoo.com
\end{center}

\smallskip
 \maketitle
\smallskip
\begin{abstract}
In this study, we treat intrinsic Finsler geometry using the Klein-Grifone approach (KG-approach).  A uniqueness and existence theorem   for the Hashiguchi connection on a Finsler manifold is  investigated intrinsically (in coordinate-free fashion). Calculations are made for the Hashiguchi connection's torsion and curvature tensors.   Some properties are examined, together with the Bianchi identities of the associated curvature and torsion tensors.  An overview of the four fundamental  linear connections in Finsler geometry in the KG-approach is provided globally for comparison's sake and completeness.
\end{abstract}

\vspace{7pt}



\section{Introduction}

~\par
Linear connections are fundamental tools in various areas of mathematics and physics, especially, in Finsler geometry and general relativity. In Finsler geometry, the theory of linear connections provide a framework for studying the geometry of spaces where the length of a curve depends not only on  position but also on  direction. In general relativity, the Levi-Civita connection is a specific type of linear connection that describes the curvature of spacetime, which is essential for understanding gravity.\\

 The most poplar  and commonly used approaches to intrinsic  Finsler
geometry are and the pullback (PB-) approach (cf.
\cite{r58,r61,r74, r44})and the Klein-Grifone (KG-) approach (cf. \cite{r21,
r22,r27, szilasi1,Chern}. Even though there are certain relations between the two approaches, each has its own  geometry that is very different from the other.\\

There exists a canonical linear connection in Riemannian geometry on a manifold $M$, and a similar canonical linear connection in Finsler geometry exists due to E. Cartan. Nevertheless, this Cartan  connection is defined on double tangent bundle of $M$ (in the KG-approach) or on the pullback bundle of $M$ (in the PB-approach).\\

By the four fundamental Finsler connections, we mean the connections that   are   introduced  by Berwald, Cartan, Chern, and Hashiguchi in local Finsler geoemetry. Soleiman \cite{amr} has studied    the four fundamental linear connections on a Finsler manifold  in the PB-approach of global Finsler geometry.  Namely, he established the existence and uniqueness theorems for these connections and calculated their associated torsion and curvature tensors. Also, he investigated some applications and changes of a Finsler manifold. On the other hand, Grifone \cite{r22} investigated the Cartan and Berwald connections in the KG-approach. In \cite{Chern}, N. L. Youssef and Elgendi addressed the existence and uniqueness theorem of the Chern connection employing the KG-approach, as well as the related curvature and torsion tensors and their properties.      The Chern and Hasiguchi connections were studied by Szilasi and Vincze \cite{szilasi}, although they did so by lifting vector fields to the tangent bundle. The KG-approach has not, to the best of our knowledge, established the existence and uniqueness theorems for the Hashiguchi connection from a purely global standpoint.\\

 In this investigation, we study the existence and uniqueness  of Hashiguchi connection's theorem on a Finsler manifold, applying the KG-approach to Finsler geometry. The formulae for this connection's curvature and torsion tensors are established.
Furthermore, we prove certain properties and the Bianchi identities of the curvature and torsion tensors.

\medskip

The following is the paper's structure. We provide the material that will be needed for the duration of the current work in the first section. We provide a brief overview of the Fr\"{o}licher-Nijenhuis formalism pertaining to the vector forms and derivation, as well as the basics of KG-approach to intrinsic Finsler geometry.
 Aside from certain fundamentals regarding Berwald and Chern connections, we focus on the most significant features and equations related to the curvature tensors of Cartan connection  in the second section.
 We establish the Hashiguchi connection's existence and uniqueness theorem in the third section. We obtain the equations for this connection's curvature and torsion tensors. We further investigate specific features of the curvature tensors of the Hashiguchi connection as well as the Bianchi identities.

 \medskip

To provide a complete picture, we summarize the four essential linear connections in Finsler geometry as defined in the KG-approach. Additionally, an appendix offers detailed local formulas and comparisons with the PB-approach, aiding in a deeper understanding of these connections.

\section{Preliminaries}
This section provides a basic introduction to the KG-approach, a framework for studying global Finsler geometry. For a more comprehensive understanding, please consult \cite{r21,r22,r27}. Throughout this paper, we will assume that all geometric objects are infinitely differentiable.

The $n$-dimensional smooth manifold $M$ is denoted throughout, and the $\Real$-algebra of $C^\infty$-functions on the manifold $M$ is denoted by $C^\infty(M)$.  Moreover, the $C^\infty (M)$-module of vector fields on $M$ is denoted as $\mathfrak{X}(M)$.    The subbundle of nonzero vectors tangent to $M$ is $\T M$, and $TM$ is the tangent bundle of $M$. We denote the vertical subbundle by $V(TM)$.  The exterior derivative of $f$ is $df$, while the derivative $ i_{\eta }$ is the interior product corresponding to $\eta  \in\mathfrak{X}(M)$. In addition, for    a vector form $K$ we have the derivative $ d_{K}:=[i_{K},d]$. As a special case, $\mathcal{L}_\eta $  is the Lie derivative in direction of $\eta \in\mathfrak{{X}}(M)$.

The double tangent bundle $T(\T M)$ and the pullback bundle $\pi^{-1}(TM)$ are related by  short exact sequence:
$$0\longrightarrow
 \pi^{-1}(TM)\stackrel{\gamma}\longrightarrow T(\T M)\stackrel{\rho}\longrightarrow
\pi^{-1}(TM)\longrightarrow 0 ,\vspace{-0.1cm}$$
 where $\rho := (\pi_{\T M},\pi)$ and $\gamma (u,v):=j_{u}(v)$ define the bundle morphisms $\rho$ and $\gamma$, respectively, and $j_{u}$ is the natural isomorphism $j_{u}:T_{\pi_{M}(v)}\longrightarrow T_{u}(T_{\pi_{M}(v)}M)$.  The natural almost tangent structure of $T M$ is the vector $1$-form $J$ on $TM$ defined by $J:=\gamma\circ\rho$. The canonical (Liouville) vector field, or fundamental vector field on $TM$,   is  ${C}:=\gamma\circ\overline{\eta}$, where $\overline{\eta}$ is the vector field on $\pi^{-1}(TM)$ defined by $\overline{\eta}(u)=(u,u)$.

We will require the Fr\"{o}licher-Nijenhuis bracket evaluation in specific particular situations for this work \cite{r20}:

 Let $L$ be  a vector $\ell$-form, then for     $\eta \in \mathfrak{X}(M)$, and for all $\eta _{1}, ..., \eta _{\ell}\in
\mathfrak{X}(M)$, we have
$$[\eta ,L](\eta_{1},...,\eta _{\ell})=[\eta , L(\eta _{1},...,\eta _{\ell})]-\sum_{i=1}^{\ell}
L(\eta _{1},..., [\eta ,\eta _{i}],...,\eta _{\ell}).$$
Particularly, for  a vector $1$-form $L$, we get
$$[\zeta ,L]\xi =[\zeta , L\xi ]-L[\zeta ,\xi ].$$
 In addition, for vector $1$- forms $K, L$, then  for all $\zeta ,
\eta \in \mathfrak{X}(M)$,
\begin{eqnarray*}
  [K,L](\zeta ,\eta )&=& [K\zeta ,L\eta ]+[L\zeta ,K\eta ]+K L[\zeta ,\eta ]+L K[\zeta ,\eta ] \\
   && -K[L \zeta ,\eta ]-K[ \zeta ,L \eta ]-L[K \zeta ,\eta ]-L[ \zeta ,K \eta ].
\end{eqnarray*}
Now, for a vector $1$-form $K$, the Nijenhuis torsion of   $K$ is the vector $2$-form $N_{K}:=\frac{1}{2}[K,K]$ and given by
\begin{equation}\label{Nk}
     N_{K}:= \frac{1}{2}[K,K](\zeta ,\xi )=[K\zeta ,K\xi ]+K^{2}[\zeta ,\xi ] -K[K \zeta ,\xi ]-K[ \zeta ,K
     \xi ].
\end{equation}
The following features can be demonstrated for the natural almost tangent structure $J$:
\begin{equation}\label{JJ}
  [J,J]=0, \ \,  J^{2}=0 \ \  [{C}, J] = -J,
  \ \ \text{Im}(J) = Ker (J) = V (T M),
\end{equation}
If $i_{J\eta }\omega=0$ for all $  \eta \in \cppp$, then this indicates that a scalar $p$-form $\omega$ is semi-basic. Also, a vector $\ell$-form $K$ is semi-basic if and only if $ JK=0\,\, \text{and}\,\, i_{J\eta }K=0, $ for all $  \eta \in \cppp$.
 In the case where $\mathcal{L}_C\omega=r\omega$, a scalar $\ell$-form $\omega$ is homogeneous of degree r.
  For every vector $\ell$-form $L$ such that $[C,L]=(r-1)L$, the vector is homogeneous of degree $r$, or $h(r)$. That is,  $J$ is $h(0)$.

A a semispray $S$ on $M$ is a vector field   on $TM$ with the properties that $S$ is  $C^{\infty}$ on $\T M$,   $C^{1}$ on $TM$, and    $JS = C$. A spray  is a semispray  $S$ as well as it is  homogeneous of degree $2$, i.e, $[{C},S]= S $.

If vector $1$-form $\Gamma$  on $TM$ satisfies that $J \Gamma=J, \ \Gamma J=-J$,  smooth on $\T M$, and $C^{0}$ on $TM$, then    $\Gamma$  constitutes a nonlinear connection on the manifold  $M$.
The definitions of the vertical projectors $v$ and the horizontal   $h$ attached to $\Gamma$ are provided as follows: $$v:=\frac{1}{2} (I-\Gamma),\, h:=\frac{1}{2} (I+\Gamma).$$
As a result, we have the direct sum decomposition, generated by $\Gamma$, of the double tangent bundle $TTM$ as follows
$$TT M=V(TM)\oplus H(TM)$$
where $H(TM)$ refers to the horizontal subbundle (or simply, bundle) $V(TM)$ the vertical bundle. Moreover,  the horizontal bundle induced by $\Gamma$ is given by $H(TM):=Im \, h = \ker\,v $ and, on the other hand,  the vertical bundle is given by $\,\,V(TM):= Im \, v=Ker \, h$. We shall refer to the elements of $V(TM)$ (resp. $H(TM)$) as $v\zeta $ (resp. $h\zeta $). We have the properties  $J v=0, \,\,\, v J=J, \,\,\, J h =J, \,\,\, h J=0$. Moreover, if $[{C},\Gamma]=0$, then $\Gamma$   is homogeneous.

A semispray $S$ that is horizontal with regard to $\Gamma$ can be associated with each nonlinear connection $\Gamma$. This semispray is denoted by $S=hS'$, where $S'$ is an arbitrary semispray. Furthermore, a semispray  attached to a homogeneous $\Gamma$ is a spray.

The torsion   of a nonlinear connection   $\Gamma$  can defined by  a vector $2$-form on $TM$, precisely,   $t:=\frac{1}{2} [J,\Gamma]$.
Moreover, the vector $2$-form, denoted by $\mathfrak{R}:=-\frac{1}{2}[h,h]$, gives the curvature of $\Gamma$. Now, the properties  $FJ=h$ and $Fh=-J$ provide the so-called  almost complex structure $F$ $(F^2=-I)$ corresponding to $\Gamma$. For any $z\in TM$, this $F$ provides an isomorphism of $T_zTM$.

\begin{defn}\label{Finsler} \emph{\cite{r27}} A Finsler  manifold (space)
 of dimension $n$ is the pair $(M,E)$, where $M$ is an $n$-dimensional
 smooth manifold,  and $E$ is the function (called the energy function)
 $$E: TM \To \Real   $$
 with     the conditions:
 \begin{description}
    \item[ {(a)} ] $E$ is positive, i.e., $E(z)>0 $ for all $z\in \T M$ and $E(0)=0$, $\T M=TM\backslash \{0\}$,
    \item[ {(b) }] $E$ is smooth  on the slit tangent bundle   $\T M$, and $C^{1}$ on $TM$,
    \item[ {(c)} ] $E$ is positively homogeneous of degree $2$ in the directional variable, i.e.,
    $\mathcal{L}_{{C}} E=2E$,
    \item[ {(d) }] The   fundamental  $2$-form
    $\Omega:=dd_{J}E$       has maximal rank.
     \end{description}
 \end{defn}
 As shown in \cite{r27}, we have the following result:
\begin{thm}{\em{\cite{r27}}}\label{spray} Consider a Finsler space $(M,E)$. The Euler-Lagrange equation $i_{S}\Omega = -dE$ determines  a unique vector field $S\in \cppp$   which is a spray.  Furthermore, this spray is known as the canonical spray (or the geodesic spray)  of the Finsler manifold $(M,E)$.
 \end{thm}

 Building upon the work of \cite{r27}, we provide an essential result concerning the existence and uniqueness of a particular nonlinear connection endowed with remarkable properties.
\begin{thm}{\em{\cite{r27}}}\label{Barthel} There is a  unique conservative and  homogeneous nonlinear connection on a Finsler manifold $(M,E)$, that is,$d_hE=0$. Moreover, this connection has zero torsion and   given by  $$\Gamma = [J,S] . $$
Where $S$ is the   canonical spray of $E$.
\end{thm}
The connection we have discussed, which is also known as the Cartan nonlinear connection, the canonical connection, or the Barthel connection, is fundamental to the Finsler manifold $(M,E)$. It's worth noting that the canonical spray is a specific type of spray associated with the Barthel connection, known as a semi-spray

\section{Berwald, Cartan, and Chern connections}

~\par
We present essential background information on the relationship between Berwald and Cartan connections, which is pertinent to the current investigation. For a more comprehensive treatment, the reader is referred to \cite{r22} and \cite{Nabil.1}.

\begin{thm}\label{Th:Berwald_connec.} \cite{r22} On  a Finsler manifold   $(M,E)$,
we have a unique linear connection  \, $\overcirc{D}$ on $TM$  that satisfies the   facts:
\begin{description}
                  \item[(a)]$\overcirc{D}J=0$.\hspace{5.4cm}\em{\textbf{(b)}}\, $\overcirc{D}C=v$.
                  \item[(c)] $\overcirc{D}\Gamma=0\,\,
                  (\Longleftrightarrow\overcirc{D}h=\overcirc{D}v=0  )$. \hspace{1.7cm}{\textbf{(d)}}\, $\overcirc{D}_{J\zeta }J\xi =J[J\zeta ,\xi ]$.
                  \item[(e)]$\overcirc{T}(J\zeta ,\xi )=0$,
 \end{description}
 where the Barthel connection's horizontal and vertical projectors are denoted by $h$ and $v$. The (classical) torsion of \, $\overcirc{D}$ is \, $\overcirc{T}$, and $\Gamma=[J,S]$. We refer to this connection as the Berwald connection.

\end{thm}

The explicit formulae   of Berwald connection \,  $\overcirc{D}$ are given as follows:
\begin{equation}\label{berwaldconn.}
  \left.
    \begin{array}{rcl}
  \overcirc{D}_{J\zeta }J\xi &=&J[J\zeta ,\xi ],\\
\overcirc{D}_{h\zeta }J\xi &=&v[h\zeta ,J\xi ],\\
  \overcirc{D}F&=&0,
 \end{array}
  \right\}
\end{equation}
where $F$ is the corresponding  almost complex structure to  the Barthel connection $\Gamma$.
\begin{lem}
For the  Berwald connection, we have   the property
$$\overcirc{T}(h\zeta ,h\xi )=\mathfrak{R}(\zeta ,\xi ),$$
where $\mathfrak{R}$ is the curvature of the Barthel connection.
\end{lem}

 Let $(M,E)$  be a  Finsler manifold   equipped  with the fundamental form  $\Omega =dd_{J}E$. Then, the map $\overline{g}$ given by
$$\overline{g}(J \zeta ,J \xi ):=\Omega(J\zeta ,\xi ), \ \forall \ \zeta , \xi  \in  T(TM)$$
presents a  metric  on $V(TM)$. Moreover, the metric  $\overline{g}$   can be extended to a metric $g$ on $T(TM)$   by the formula:

 \begin{equation}\label{metricg}
 g(\zeta ,\xi )=\overline{g}(J\zeta ,J\xi )+\overline{g}(v\zeta ,v\xi )=\Omega(\zeta ,F\xi ).
\end{equation}

According to \cite{r22}, we have the following theorem which characterizes  the Cartan  connection on a Finsler manifold $(M,E)$.
\begin{thm}\label{Th:Cartan_connec.}\cite{r22} Assume that   $(M,E)$ is a Finsler manifold. Then, we have a unique  linear connection ${D}$ on $TM$ such that  the following properties are satisfied:
\begin{description}
                  \item[(a)]${D}J=0$.\hspace{4.4cm} \em{\textbf{(b)}} ${D}C=v$.
                  \item[(c)] ${D}\Gamma=0 \,\,(\Longleftrightarrow {D}h={D}v=0 )$.\hspace{.8cm}\textbf{(d)} ${D}g=0$.
                  \item[(e)] ${T}(J\zeta ,J\eta )=0$.\hspace{3.25cm}\textbf{(f)} $JT(h\zeta ,h\xi )=0$.
 \end{description}
 \end{thm}
The   connection ${D}$ mentioned above is referred as the Cartan  connection.  Moreover, the explicit formulae  of\,  ${D}$ are given by:
\begin{equation}\label{cartanconn.}
  \left.
    \begin{array}{rcl}
  D_{J\zeta }J\xi &=&\overcirc{D}_{J\zeta }J\xi +\mathcal{C}(\zeta ,\xi ),\\
  D_{h\zeta }J\xi &=&\overcirc{D}_{h\zeta }J\xi +\mathcal{C}'(\zeta ,\xi ),\\
  {D}F&=&0,
\end{array}
  \right\}
\end{equation}
where $\mathcal{C}$ and $\mathcal{C}'$ are   scalar 2-forms on $TM$ defined  by the formulae
$$\Omega(\mathcal{C}(\zeta ,\eta ),\xi )=\frac{1}{2}(\mathcal{L}_{J\zeta }(J^\ast g))(\eta ,\xi ),\quad\quad
\Omega(\mathcal{C}'(\zeta ,\eta ),\xi )=\frac{1}{2}(\mathcal{L}_{h\zeta }g)(J\eta ,J\xi ), $$
where we use  $(J^\ast g)(\eta ,\xi )=g(J\eta ,J\xi )$.
It should be noted that  ${\mathcal{C}}$ and $\mathcal{C}'$ are   the first and the second Cartan tensors respectively. Moreover,  ${\mathcal{C}}$ and $\mathcal{C}'$     are semi-basics,
 symmetric,  and
  \begin{equation}\label{c(s)}
  {\mathcal{C}}(\eta ,S)=\mathcal{C}'(\eta ,S)=0.
  \end{equation}
Recently, in \cite{Chern}, the Chern connection was studied and characterized by the following theorem.
\begin{thm}\label{Th:Chern_connec.}\cite{Chern} Suppose $(M,E)$ is a Finsler manifold. Then, we have    a unique linear connection, denoted by \, $\overast{D}$,  on $TM$ such  that  the following facts are attained:
\begin{description}
                  \item[(a)]$\overast{D}J=0$.\hspace{4.4cm} \em{\textbf{(b)}} $\overast{D}C=v$.
                  \item[(c)] $\overast{D}\Gamma=0 \,\,(\Longleftrightarrow \overast{D}h=\overast{D}v=0 )$.\hspace{1cm}\textbf{(d)} $\overast{D}_{h\zeta }g=0$.
                  \item[(e)] $\overast{T}(J\zeta ,J\xi )=0$.\hspace{3.25cm}\textbf{(f)} $J\ \overast{T}(h\zeta ,h\xi )=0$.
 \end{description}
 \end{thm}
  The  connection  \, $\overast{D}$  is called the Chern  connection and its    formulae    are characterized  by:
\begin{equation}\label{cartanconn.}
  \left.
    \begin{array}{rcl}
  \overast{D}_{J\zeta }J\xi &=&\overcirc{D}_{J\zeta }J\xi ,\\
  \overast{D}_{h\zeta }J\xi &=&\overcirc{D}_{h\zeta }J\xi +\mathcal{C}'(\zeta ,\xi ),\\
  \overast{D} F&=&0.
\end{array}
  \right\}
\end{equation}
For the purpose of our subsequent use, we provide the following lemmas:
  \begin{lem} For  Cartan connection, the (h)h-torsion ${T}(h\zeta ,h\xi )$ and (h)v-torsion $T(h\zeta ,J\xi )$    can be calculated  by
$${T}(h\zeta ,h\xi )=\mathfrak{R}(\zeta ,\xi ),\quad T(h\zeta ,J\xi )=(\mathcal{C}'-F\mathcal{C})(\zeta ,\xi ),$$
where $\mathfrak{R}$ is the curvature of the Barthel connection.
\end{lem}
\begin{lem}\label{car.r,p,q} For  Cartan connection, the h-curvature $R$, hv-curvature $P$,   and v-curvature $Q $ are calculated  as follows:
\begin{description}
  \item[(a)]$ R(\eta ,\kappa )\xi   = \overcirc{R}(\eta ,\kappa )\xi +(D_{h\eta }\mathcal{C}')(\kappa ,\xi )-(D_{h\kappa }\mathcal{C}')(\eta ,\xi )
      +\mathcal{C}'(F\mathcal{C}'(\eta ,\xi ),\kappa )\\
      {\hspace{1.9cm}}-\mathcal{C}'(F\mathcal{C}'(\kappa ,\xi ),\eta )+\mathcal{C}(F\mathfrak{R}(\eta ,\kappa ),\xi ). $
  \item[(b)] $ P(\eta ,\kappa )\xi   = \overcirc{P}(\eta ,\kappa )\xi +(D_{h\eta }{\mathcal{C}})(\kappa ,\xi )-(D_{J\kappa }\mathcal{C}')(\eta ,\xi )
      +\mathcal{C}(F\mathcal{C}'(\eta ,\xi ),\kappa ) \\
      {\hspace{1.9cm}} +{\mathcal{C}}(F\mathcal{C}'(\eta ,\kappa ),\xi )-\mathcal{C}'(F{\mathcal{C}}(\kappa ,\xi ),\eta )-\mathcal{C}'(F{\mathcal{C}}(\eta ,\kappa ),\xi ). $
  \item[(c)] $Q(\eta ,\kappa )\xi ={\mathcal{C}}(F{\mathcal{C}}(\eta ,\xi ),\kappa )-{\mathcal{C}}(F{\mathcal{C}}(\kappa ,\xi ),\eta ),$
\end{description}
where\, $\overcirc{R}$ and\, $\overcirc{P}$ are  the h-curvature and hv-curvature  of Berwald connection, respectively.
\end{lem}
\begin{lem}\label{car.curv.}
The curvatures of the Cartan connection $D$ have the following properties:
\begin{description}
  \item[(a)]$R(\eta ,\kappa )S=\mathfrak{R}(\eta ,\kappa )$.
  \item[(b)] $P(\eta ,\kappa )S=\mathcal{C}'(\eta ,\kappa )$.
  \item[(c)]$P(S,\eta )\kappa =P(\eta ,S)\kappa =0$.
  \item[(d)]$Q(S,\eta )\kappa =Q(\eta ,S)\kappa =Q(\eta ,\kappa )S=0.$
\end{description}
\end{lem}
\begin{lem}\label{jx}
A semi spray $S$ satisfies the following property
$$J[J\eta ,S]=J\eta ,\,\, \forall \eta \in \cppp.$$
\end{lem}

\begin{lem}\label{chx}
For a homogeneous connection $\bf\Gamma$, its horizontal projector $\textbf{h}$ satisfies
 $$[C,\textbf{h}\zeta ]=\textbf{h}[C,\zeta ],\,\, \forall \zeta \in \cppp.$$
\end{lem}
\begin{proof}
Since $\bf\Gamma$ is homogeneous, then $\textbf{h}$ is h(1). Thus, $[C,\textbf{h}]=0$ and hence
\begin{eqnarray*}
   0&=&[C,\textbf{h}]\zeta   \\
   &=&[C,\textbf{h}\zeta ]-\textbf{h}[C,\zeta ].
\end{eqnarray*}
\end{proof}

\section{Hashiguchi connection}

~\par
This section is dedicated to the exploration of the existence, uniqueness, and properties of the Hashiguchi connection. Explicit formulas for the torsion and curvature tensors of this connection are derived, and Bianchi identities are investigated.

For this purpose, we list the following definitions which are essential in the KG-approach to the theory of connections in the study of intrinsic Finsler geometry.
\begin{defn}
A linear connection $\textbf{D}$ on $TM$ is called  regular if $\textbf{D}J=0$ and the map $$\varphi:V(TM)\rightarrow V(TM),$$
 given by $\zeta \rightarrow \textbf{D}_\zeta C$, defines an isomorphism on $V(TM)$.
\end{defn}
The above map $\varphi$  can be considered as a restriction to $V(TM)$ of a map $\widetilde{\varphi}=\textbf{D}_\zeta C$. For a regular
connection $\textbf{D}$ on $TM$ there is corresponding  connection $\bf{\Gamma}$  on $M$ given  by $${\bf{\Gamma}={I-2\varphi^{-1}}\circ
 \textbf{D}} {C}, $$
where $\bf{\Gamma}$ is called  induced by $\textbf{D}$.

\begin{defn}
Consider  a regular  connection $\textbf{D}$    on $TM$ and   the connection $\bf{\Gamma}$   induced on $M$ by $\textbf{D}$. Then, $\textbf{D}$ is
 called  reducible if $\bf{D\Gamma}=0.$
\end{defn}

\begin{defn}
We call  a linear  connection $\textbf{D}$ on $TM$   almost-projectable  if $\textbf{D}{J}=0$ and $\textbf{D}_{J\eta }C=J\eta $, for all $\eta \in T(TM)$.
\end{defn}
If we replace the axiom $\textbf{D}_{J\zeta }C=J\zeta $ with the more general one, that is, $\textbf{D}_{J\eta }J\xi =J[J\eta ,\xi ]$, for all $\eta ,\, \xi \in \cppp$, then the connection $\textbf{D}$ is called normal almost-projectable.

We will refer to the connection $\bf{\Gamma}$ on $M$ induced by the almost-projectable (resp. normal almost-projectable) connection $\textbf{D}$ on $TM$ as the projection of $\textbf{D}$ . More precisely, we will say that $\textbf{D}$  projects (resp. projects normally) onto $\bf{\Gamma}$.

\begin{defn}
Assume that $\bf{\Gamma}$ is   connection on $M$. Then, a reducible connection $\textbf{D}$ on $TM$ that projects on $\bf{\Gamma}$ is the lift of $\bf{\Gamma}$. If $\textbf{D}$ is normal, then the lift of $\bf{\Gamma}$ is considered normal.

\end{defn}

\begin{defn}
 Let a linear connection $\textbf{D}$ on $TM$. $\textbf{D}$ is said to be horizontally metric or (h-metrical) if it satisfies that
 $\textbf{D}_{h\zeta }g=0$, for all $\zeta  \in \cppp$.
\end{defn}

\begin{lem}\label{df=0}
Consider a reducible connection $\textbf{D}$  with the almost-complex structure $\textbf{F}$    associated to
 the connection $\bf{\Gamma}$ induced by $\textbf{D}$. Then,  we have $\textbf{DF}=0$.
\end{lem}
\begin{proof}
Since $\textbf{D}$ is reducible, then we have $\bf{D\Gamma}=0$. Thus, we get $\textbf{Dh}=\textbf{Dv}=0$, where $\textbf{h}$
 and $\textbf{v}$ are the associated horizontal   vertical projectors to $\bf{\Gamma}$. Therefore, using the facts that $\bf{FJ=h}$, $\bf{Fh=-J}$, and $\bf{JF=v}$, we obtain
\begin{eqnarray*}
  \textbf{F}(\textbf{D}\zeta ) &=&\textbf{F}(\textbf{Dh}\zeta +\textbf{Dv}\zeta )  \\
   &=&\textbf{F}(\textbf{hD}\zeta +\textbf{JDF}\zeta )\\
   &=&-\textbf{JD}\zeta +\textbf{hDF}\zeta \\
   &=&\textbf{D}(-\textbf{J}\zeta )+\textbf{DhF}\zeta \\
   &=&\textbf{DFh}\zeta +\textbf{DFv}\zeta \\
   &=&\textbf{DF}(\textbf{h}\zeta +\textbf{v}\zeta )\\
   &=&\textbf{DF}\zeta .
\end{eqnarray*}
Hence, $\textbf{DF}=0$.
\end{proof}

\begin{lem}\label{cdash}
The second Cartan tensor $\mathcal{C}'_b(\zeta ,\kappa ,\xi ):=g(\mathcal{C}'(\zeta ,\kappa ),J\xi )$ is completely  symmetric.
\end{lem}

We are now prepared to present the existence and uniqueness theorem for the Hashiguchi connection on $TM$.

\begin{thm}\label{hashiconnc.}    Consider the    Finsler manifold $(M,E)$,   then, we have  a unique
lift  $\,\overdiamond{D}$ of the Barthel connection $\Gamma=[J,S]$ such that the following assertions are satisfied:
\begin{description}
  \item[(a)] $\,\,\overdiamond{D}$ is vertically metric: $\,\,\overdiamond{D}_{J\eta }g=0,\,$ for all $ \eta \in \cppp$.
  \item[(b)] The classical torsion $\,\overdiamond{T}(J\zeta ,J\eta )$ satisfies that: $\,\overdiamond{T}(J\zeta ,J\eta )=0, \, $ for all $ \zeta ,\eta \in \cppp$.
  \item[(c)] The classical torsion $\,\overdiamond{T}(h\zeta ,J\eta )$ satisfies that: $v\,\,\overdiamond{T}(h\zeta ,J\eta )=0, \,$ for all $ \zeta ,\eta \in \cppp$.
\end{description}
\end{thm}
The  connection $\,\overdiamond{D}$ mentioned above   is called the Hashiguchi connection.

\begin{proof} We start by  proving the \textbf{uniqueness}. Let $\zeta $, $\eta $ and $\xi $ $\in\cppp$.
Since  $\,\,\overdiamond{D}$ is a lift of the Barthel connection $\Gamma=[J,S]$, then
 $\,\,\overdiamond{D}\Gamma=0$, $\,\,\overdiamond{D} J=0$ and thus, by Lemma \ref{df=0}, we have
 \begin{equation}\label{dhashi.f=0}
 \,\,\overdiamond{D} F=0.
 \end{equation}
 Making  use of the conditions $(\,\,\overdiamond{D}_{J\zeta }g)(J\eta ,J\xi )=0$ and $\,\overdiamond{T}(J\zeta ,J\eta )=0$, we get:
\begin{eqnarray}
\label{1} J\zeta \cdot g(J\eta ,J\xi )  &=&g(\,\,\overdiamond{D}_{J\zeta }J\eta ,J\xi )+g(J\eta ,\,\,\overdiamond{D}_{J\zeta }J\xi )  \\
\label{2}   J\eta \cdot g(J\xi ,J\zeta )  &=&g(\,\,\overdiamond{D}_{J\eta }J\xi ,J\zeta )+g(J\xi ,\,\,\overdiamond{D}_{J\eta }J\zeta ) \\
\label{3}   J\xi \cdot g(J\zeta ,J\eta )  &=&g(\,\,\overdiamond{D}_{J\xi }J\zeta ,J\eta )+g(J\zeta ,\,\,\overdiamond{D}_{J\xi }J\eta ).
\end{eqnarray}
By adding  (\ref{1}), (\ref{2}) and subtracting  (\ref{3}), we get
\begin{eqnarray}\label{4}
\nonumber J\zeta \cdot g(J\eta ,J\xi )&+&J\eta \cdot g(J\xi ,J\zeta ) - J\eta \cdot g(J\xi ,J\zeta )= g(\,\,\overdiamond{D}_{J\zeta }J\eta +\,\,\overdiamond{D}_{J\eta }J\zeta ,J\xi )  \\
   && +g(J\eta ,\,\,\overdiamond{D}_{J\zeta }J\xi -\,\,\overdiamond{D}_{J\xi }J\zeta )+g(\,\,\overdiamond{D}_{J\eta }J\xi -\,\,\overdiamond{D}_{J\xi }J\eta ,J\zeta ).
\end{eqnarray}
By the condition $\,\overdiamond{T}(J\zeta ,J\eta )=0$, we have
\begin{equation}\label{jt}
\,\,\overdiamond{D}_{J\zeta }J\eta -\,\,\overdiamond{D}_{J\eta }J\zeta =[J\zeta ,J\eta ].
\end{equation}
From  (\ref{4}) and (\ref{jt}), we get
\begin{eqnarray}\label{gnabla}
\nonumber  g(2\,\,\overdiamond{D}_{J\zeta }J\eta ,J\xi )  &=&J\zeta \cdot g(J\eta ,J\xi )+J\eta \cdot g(J\xi ,J\zeta ) - J\eta \cdot g(J\xi ,J\zeta )\\
&&{\hspace{-2cm}}+g([J\zeta ,J\eta ],J\xi )- g([J\zeta ,J\xi ],J\eta )-g([J\eta ,J\xi ],J\zeta )
\end{eqnarray}
By the facts that $\Omega(\zeta ,\eta )=g(\zeta ,J\eta )-g(J\zeta ,\eta )$ and     $J{\mathcal{C}}=0$, we have  $$\frac{1}{2}(\mathcal{L}_{J\zeta }(J^\ast g))(\eta ,\xi )=\Omega({\mathcal{C}}(\zeta ,\eta ),\xi )=g({\mathcal{C}}(\zeta ,\eta ),J\xi )
={\mathcal{C}}_b(\zeta ,\eta ,\xi ),$$ which is completely  symmetric. Now,
\begin{eqnarray*}
  g(2{\mathcal{C}}(\zeta ,\eta ),J\xi ) &=& J\zeta \cdot g(J\eta ,J\xi )-g(J[J\zeta ,\eta ],J\xi )-g(J\eta ,J[J\zeta ,\xi ]),\\
  g(2{\mathcal{C}}(\eta ,\xi ),J\zeta ) &=& J\eta \cdot g(J\xi ,J\zeta )-g(J[J\eta ,\xi ],J\zeta )-g(J\xi ,J[J\eta ,\zeta ]),\\
  -g(2{\mathcal{C}}(\xi ,\zeta ),J\eta ) &=&- J\xi \cdot g(J\zeta ,J\eta )+g(J[J\xi ,\zeta ],J\eta )+g(J\zeta ,J[J\xi ,\eta ]).
\end{eqnarray*}
By adding the  three equations above, we get
\begin{eqnarray}\label{cdash}
  \nonumber  g(2{\mathcal{C}}(\zeta ,\eta ),J\xi )&=&J\zeta \cdot g(J\eta ,J\xi )+J\eta \cdot g(J\xi ,J\zeta )-J\xi \cdot g(J\zeta ,J\eta )\\
 \nonumber  && -g(J[J\zeta ,h\eta ]+J[J\eta ,h\zeta ],J\xi )+g(J[J\xi ,h\zeta ]-J[J\zeta ,h\xi ],J\eta )\\&&
    +g(J[J\xi ,h\eta ]-J[J\eta ,h\xi ],J\zeta )
\end{eqnarray}
One can see that  (\ref{gnabla}) and (\ref{cdash}) imply
\begin{eqnarray}\label{gnapla}
 \nonumber  g(2\,\,\overdiamond{D}_{J\zeta }J\eta ,J\xi )&=&g(2{\mathcal{C}}(\zeta ,\eta ),J\xi )-g([J\eta ,J\xi ]+J[J\xi ,h\eta ]-J[J\eta ,h\xi ],J\zeta )\\
 \nonumber &&- g([J\zeta ,J\xi ]+J[J\xi ,h\zeta ]-J[J\zeta ,h\xi ],J\eta ) \\
   &&+ g([J\zeta ,J\eta ]+J[J\zeta
   ,h\eta ]+J[J\eta ,h\zeta ],J\xi ).
\end{eqnarray}
Since $J$ satisfy the property  $$[J\zeta ,J\eta ]=J[J\zeta ,h\eta ]+J[h\zeta ,J\eta ],$$ then, we have
\begin{equation}\label{dhashi.jj}
\,\,\overdiamond{D}_{J\zeta }J\eta =J[J\zeta ,\eta ]+{\mathcal{C}}(\zeta ,\eta ).
\end{equation}

Now, using the condition $v\,\,\overdiamond{T}(h\zeta ,J\eta )=0$, we have
$$v\,\,\overdiamond{D}_{h\zeta }J\eta -v\,\,\overdiamond{D}_{J\eta }h\zeta -v[h\zeta ,J\eta ]=0.$$
By the above equation and using that $vJ=J$ and $vh=0$, then we get
\begin{equation}\label{dhashi.hj}
    \overdiamond{D}_{h\zeta }J\eta =v[h\zeta ,J\eta ].
\end{equation}
The right hand side of (\ref{dhashi.jj}) and (\ref{dhashi.hj}) are uniquely and hence $\,\,\overdiamond{D}_\zeta \eta $ is uniquely determined by
(\ref{dhashi.f=0}), (\ref{dhashi.jj})  and (\ref{dhashi.hj}).\\

To prove   the \textbf{existence} of $\,\,\overdiamond{D}$, we define $\,\,\overdiamond{D}$    by  the requirement that   (\ref{dhashi.f=0}), (\ref{dhashi.jj}) and
(\ref{dhashi.hj}) hold for all $\zeta ,\eta  \in \cppp$.

Now, we have to establish  the following properties. $\,\,\overdiamond{D}$ is  lift of $\Gamma=[J,S]$: ( $\,\overdiamond{D} J=0$, $\,\,\overdiamond{D} C=v$, $\,\,\overdiamond{D} \Gamma=0$).

\noindent $\bullet$ $\,\,\overdiamond{D} J=0$, it is sufficient to show that  $J\,\,\overdiamond{D}_{\zeta }\eta =\,\,\overdiamond{D}_\zeta J\eta $. From  (\ref{dhashi.f=0}), (\ref{dhashi.jj}) and
 (\ref{dhashi.hj}), we have
\begin{eqnarray*}
   J\,\,\overdiamond{D}_{h\zeta }\eta &=& J\,\,\overdiamond{D}_{h\zeta }h\eta +J\,\,\overdiamond{D}_{h\zeta }v\eta  \\
   &=&J\,\,\overdiamond{D}_{h\zeta }h\eta +J\,\,\overdiamond{D}_{h\zeta }JF\eta   \\
   &=&J\,\,\overdiamond{D}_{h\zeta }h\eta +Jv[h\zeta ,v\eta ]  \\
   &=& JFv[h\zeta ,J\eta ]\\
   &=&v[h\zeta ,J\eta ]\\
   &=&\,\,\overdiamond{D}_{h\zeta }J\eta .
\end{eqnarray*}
Similarly, one can show that $J\,\,\overdiamond{D}_{v\zeta }\eta =\,\,\overdiamond{D}_{v\zeta }J\eta $.

\noindent$\bullet$ $\,\,\overdiamond{D} C=v$, we have to show that $\,\,\overdiamond{D}_{h\zeta }C=vh\zeta =0$ and $\,\,\overdiamond{D}_{J\zeta }C=vJ\zeta =J\zeta $ as follows. From
(\ref{dhashi.f=0}),  (\ref{dhashi.jj}),   (\ref{dhashi.hj}) and Lemma \ref{chx}, we get
\begin{eqnarray*}
   \,\,\overdiamond{D}_{h\zeta }JS&=& v[h\zeta ,JS]
   =-v[C,h\zeta ]
   =-vh[C,\zeta ]
   =0,
\end{eqnarray*}
then  by (\ref{c(s)}) and  Lemma \ref{jx}, we obtain
\begin{eqnarray*}
   \,\,\overdiamond{D}_{J\zeta }JS&=& J[J\zeta ,S]+\mathcal{C}(\zeta ,S)
   =J\zeta
   =vJ\zeta .
\end{eqnarray*}

\noindent$\bullet$ $\,\,\overdiamond{D} \Gamma=0$ or equivalently $\,\,\overdiamond{D} v=0$ or $\,\,\overdiamond{D} h=0$. We will show only that $h\,\,\overdiamond{D}_{h\zeta }\eta =\,\,\overdiamond{D}_{h\zeta }h\eta $. By  (\ref{dhashi.f=0}),  (\ref{dhashi.jj}) and (\ref{dhashi.hj}), we get
\begin{eqnarray*}
   h\,\,\overdiamond{D}_{h\zeta }\eta &=& h\,\,\overdiamond{D}_{h\zeta }h\eta +h\,\,\overdiamond{D}_{h\zeta }v\eta  \\
   &=&h\,\,\overdiamond{D}_{h\zeta }h\eta +h\,\,\overdiamond{D}_{h\zeta }JF\eta   \\
   &=&h\,\,\overdiamond{D}_{h\zeta }h\eta +hv[h\zeta ,v\eta ]  \\
   &=& hFv[h\zeta ,J\eta ]\\
   &=&Fv^2[h\zeta ,J\eta ]\\
   &=&Fv[h\zeta ,J\eta ]\\
   &=&\,\,\overdiamond{D}_{h\zeta }h\eta .
\end{eqnarray*}

\noindent$\bullet$ $\,\,\overdiamond{D}$ is h-metrical:
 $g$ is vertically metric $(\,\overdiamond{D}_{J\zeta }g)(J\eta ,J\xi )=0$. By   (\ref{dhashi.f=0}), (\ref{dhashi.jj}) and (\ref{dhashi.hj}), we have
\begin{eqnarray*}
 (\overdiamond{D}_{J\zeta }g)(J\eta ,J\xi )  &=& J\zeta .g(J\eta ,J\xi )- g(\,\overdiamond{D}_{J\zeta }J\eta ,J\xi )-g(J\eta ,\,\overdiamond{D}_{J\zeta }J\xi ) \\
   &=&  J\zeta .g(J\eta ,J\xi )- g(J[J\zeta ,\eta ]+{\mathcal{C}}(\zeta ,\eta ),J\xi )\\
   &&-g(J\eta ,J[J\zeta ,\xi ]+{\mathcal{C}}(\zeta ,\xi ))\\
   &=&J\zeta .g(J\eta ,J\xi )- g(J[J\zeta ,\eta ],J\xi )-g(J\eta ,J[J\zeta ,\xi ])  \\
   &&-2{\mathcal{C}}_b(\zeta ,\eta ,\xi )\\
   &=&0
\end{eqnarray*}

\noindent$\bullet$ $\,\overdiamond{T}(J\zeta ,J\eta )=0$,  by (\ref{dhashi.f=0}),  (\ref{dhashi.jj}) and (\ref{dhashi.hj}), we have
\begin{eqnarray*}
  \overdiamond{T}(J\zeta ,J\eta ) &=&\,\overdiamond{D}_{J\zeta }J\eta -\,\overdiamond{D}_{J\eta }J\zeta -[J\zeta ,J\eta ]  \\
   &=&J[J\zeta ,\eta ]+ {\mathcal{C}}(\zeta ,\eta )-J[J\eta ,\zeta ]- {\mathcal{C}}(\eta ,\zeta )-[J\zeta ,J\eta ]\\
   &=&0.
\end{eqnarray*}

\noindent$\bullet$ $v\,\overdiamond{T}(h\zeta ,J\eta )=0$,  by (\ref{dhashi.f=0}),  (\ref{dhashi.jj}) and (\ref{dhashi.hj}), we have
\begin{eqnarray*}
  v\,\overdiamond{T}(h\zeta ,J\eta ) &=&v\,\,\overdiamond{D}_{h\zeta }J\eta -v\,\,\overdiamond{D}_{J\eta }h\zeta -v[h\zeta ,J\eta ]  \\
   &=& \,\overdiamond{D}_{h\zeta }J\eta -v[h\zeta ,J\eta ] \\
   &=&0.
\end{eqnarray*}
Hence the proof is completed.
\end{proof}

\begin{thm} \label{hashi.nab.} The Hashiguchi connection $\,\,\overdiamond{D}$ is uniquely determined by the following relations:
\begin{description}
  \item[(a)] $\,\,\overdiamond{D}_{J\zeta }J\eta =\overcirc{D}_{J\zeta }J\eta +{\mathcal{C}}(\zeta ,\eta ).$
  \item[(b)] $\,\,\overdiamond{D}_{h\zeta }J\eta =\overcirc{D}_{h\zeta }J\eta .$
  \item[(c)] $\,\,\overdiamond{D} F=0.$
\end{description}
\end{thm}
Now, we have the enough information needed to find explicit expressions for  the attached torsion and curvature tensors to   Hashiguchi connection.
\begin{lem}For the Hashiguchi connection \,$\overdiamond{D}$, the (h)h-torsion $\,\overdiamond{T}(h\zeta ,h\xi )$,  and the  (h)v-torsion $\,\overdiamond{T}(h\zeta ,J\xi )$ are calculated as follows:
\begin{description}
  \item[(a)]$\,\overdiamond{T}(h\zeta ,h\xi )=\mathfrak{R}(\zeta ,\xi ).$
  \item[(b)]$\,\overdiamond{T}(h\zeta ,J\xi )=-F\mathcal{C}(\zeta ,\xi ).$
\end{description}
\end{lem}
\begin{proof}
~\par
\noindent \textbf{(a)} Making use of Theorem \ref{hashi.nab.},
 and the vanishing property  of the torsion $t(\zeta ,\xi )$ of the connection  $\Gamma$, that is, $0=t(\zeta ,\xi )=v[J\zeta ,h\xi ]+v[h\zeta ,J\xi ]-J[h\zeta ,h\xi ]$ together with the fact that  $hF=Fv$ , we have
\begin{eqnarray*}
  \overdiamond{T}(h\zeta ,h\xi ) &=& \,\,\overdiamond{D}_{h\zeta }h\xi -\,\,\overdiamond{D}_{h\xi }h\zeta -[h\zeta ,h\xi ] \\
   &=& Fv[h\zeta ,J\xi ]- Fv[h\xi ,J\zeta ]-[h\zeta ,h\xi ]\\
   &=&  FJ[h\zeta ,h\xi ]- Fv[J\zeta ,h\xi ]-hF[h\xi ,J\zeta ]-[h\zeta ,h\xi ]\\
    &=&  h[h\zeta ,h\xi ]-[h\zeta ,h\xi ]\\
    &=&  -v[h\zeta ,h\xi ]\\
     &=&\mathfrak{R}(\zeta ,\xi ).
\end{eqnarray*}

\noindent \textbf{(b)} By Theorem \ref{hashi.nab.} and using that the property $h[J\zeta ,v\xi ]=0$, we get
\begin{eqnarray*}
  \overdiamond{T}(h\zeta ,J\xi ) &=& \,\,\overdiamond{D}_{h\zeta }J\xi -\,\,\overdiamond{D}_{J\xi }h\zeta -[h\zeta ,J\xi ] \\
   &=& v[h\zeta ,J\xi ]- h[J\xi ,\zeta ]-F\mathcal{C}(\xi ,\zeta )-[h\zeta ,J\xi ]\\
   &=& v[h\zeta ,J\xi ]- h[J\xi ,h\zeta ]-h[h\zeta ,J\xi ]-v[h\zeta ,J\xi ]\\
    &=&-F\mathcal{C}(\zeta ,\xi ).
    \end{eqnarray*}
\end{proof}
\begin{prop}\label{hashicurv.} For     Hashiguchi connection, the h-curvature \, $\overdiamond{R}$,  the hv-curvature \,  $\overdiamond{P}$,  and the v-curvature \, $\overdiamond{Q}$, can be written  in the following formulae:
\begin{description}
  \item[(a)] $\overdiamond{R}(\zeta ,\eta )\xi =\overcirc{R}(\zeta ,\eta )\xi +\mathcal{C}(F\mathfrak{R}(\zeta ,\eta ),\xi )$.

  \item[(b)] $\overdiamond{P}(\zeta ,\eta )\xi =\overcirc{P}(\zeta ,\eta )\xi +(\,\,\overdiamond{D}_{h\zeta }\mathcal{C})(\eta ,\xi )$.

  \item[(c)] $\overdiamond{Q}(\zeta ,\eta )\xi ={Q}(\zeta ,\eta )\xi ={\mathcal{C}}(F{\mathcal{C}}(\zeta ,\xi ),\eta )-{\mathcal{C}}(F{\mathcal{C}}(\eta ,\xi ),\zeta )$.
\end{description}
\end{prop}
\begin{proof}
~\par
\noindent \textbf{(a)} By the definition of the classical curvature $K$, the properties of $F$ and the fact that $\,\,\overdiamond{D}_{h\zeta }J\eta =\,\overcirc{D}_{h\zeta }J\eta $, we have
\begin{eqnarray*}
  \overdiamond{R}(\zeta ,\eta )\xi  &=&K(h\zeta ,h\eta )J\xi   \\
   &=& \,\,\overdiamond{D}_{h\zeta }\,\,\overdiamond{D}_{h\eta }J\xi -\,\,\overdiamond{D}_{h\eta }\,\,\overdiamond{D}_{h\zeta }J\xi -\,\,\overdiamond{D}_{[h\zeta ,h\eta ]}J\xi  \\
   &=& \,\overcirc{D}_{h\zeta }\,\overcirc{D}_{h\eta }J\xi -\overcirc{D}_{h\eta }\,\overcirc{D}_{h\zeta }J\xi -\,\,\overdiamond{D}_{[h\zeta ,h\eta ]}J\xi  \\
   &=& \,\overcirc{R}(\zeta ,\eta )\xi +\,\overcirc{D}_{[h\zeta ,h\eta ]}J\xi -\,\,\overdiamond{D}_{[h\zeta ,h\eta ]}J\xi  \\
   &=& \,\overcirc{R}(\zeta ,\eta )\xi +\,\overcirc{D}_{JF[h\zeta ,h\eta ]}J\xi -\,\,\overdiamond{D}_{JF[h\zeta ,h\eta ]}J\xi  \\
   &=&\,\overcirc{R}(\zeta ,\eta )\xi -\mathcal{C}(F[h\zeta ,h\eta ],\xi )\\
   &=&\,\overcirc{R}(\zeta ,\eta )\xi +\mathcal{C}(F\mathfrak{R}(\zeta ,\eta ),\xi ).
\end{eqnarray*}
\noindent \textbf{(b)} By  Theorem \ref{hashi.nab.} and the fact that  $\,\,\overdiamond{D}_{h\zeta }J\eta =\,\overcirc{D}_{h\zeta }J\eta $, we have
\begin{eqnarray*}
  \overdiamond{P}(\zeta ,\eta )\xi  &=&K(h\zeta ,J\eta )J\xi   \\
   &=& \,\,\overdiamond{D}_{h\zeta }\,\,\overdiamond{D}_{J\eta }J\xi -\,\,\overdiamond{D}_{J\eta }\,\,\overdiamond{D}_{h\zeta }J\xi -\,\overdiamond{D}_{[h\zeta ,J\eta ]}J\xi  \\
   &=& \,\overcirc{D}_{h\zeta }(\,\overcirc{D}_{J\eta }J\xi +\mathcal{C}(\eta ,\xi ))-\overcirc{D}_{J\eta }\,\overcirc{D}_{h\zeta }J\xi -\mathcal{C}(\,\overcirc{D}_{h\zeta }\xi ,\eta )
   -\,\overdiamond{D}_{[h\zeta ,h\eta ]}J\xi \\
   &=&\,\overcirc{D}_{h\zeta }\,\overcirc{D}_{J\eta }J\xi -\,\overcirc{D}_{J\eta }\,\overcirc{D}_{h\zeta }J\xi -\,\,\overcirc{D}_{[h\zeta ,h\eta ]}J\xi
   +\, \overcirc{D}_{h\zeta }\mathcal{C}(\eta ,\xi )\\&&-\mathcal{C}(\,\overcirc{D}_{h\zeta }\xi ,\eta )- \mathcal{C}(F[h\zeta ,h\eta ],\xi )\\
   &=& \overcirc{P}(\zeta ,\eta )\xi +\, \overcirc{D}_{h\zeta }\mathcal{C}(\eta ,\xi )-\mathcal{C}(\,\overcirc{D}_{h\zeta }\eta ,\xi )-\mathcal{C}(\eta ,\,\overcirc{D}_{h\zeta }\xi ) \\
   &=&\overcirc{P}(\zeta ,\eta )\xi +(\,\overdiamond{D}_{h\zeta }\mathcal{C})(\eta ,\xi ).
\end{eqnarray*}
\noindent \textbf{(c)} By  Theorem \ref{hashi.nab.}  and the properties $\,\,\overdiamond{D}_{J\zeta }J\eta ={D}_{J\zeta }J\eta $, $hv=0$,  we get
\begin{eqnarray*}
  \overdiamond{Q}(\zeta ,\eta )\xi  &=&K(J\zeta ,J\eta )J\xi   \\
   &=& \,\,\overdiamond{D}_{J\zeta }\,\,\overdiamond{D}_{J\eta }J\xi -\,\,\overdiamond{D}_{J\eta }\,\,\overdiamond{D}_{J\zeta }J\xi -\,\,\overdiamond{D}_{[J\zeta ,J\eta ]}J\xi  \\
    &=& {D}_{J\zeta }{D}_{J\eta }J\xi -{D}_{J\eta }{D}_{J\zeta }J\xi -{D}_{[J\zeta ,J\eta ]}J\xi  \\
    &=&  {Q}(\zeta ,\eta )\xi \\
    &=& {\mathcal{C}}(F{\mathcal{C}}(\zeta ,\xi ),\eta )-{\mathcal{C}}(F{\mathcal{C}}(\eta ,\xi ),\zeta ).
\end{eqnarray*}
\end{proof}

\begin{prop}\label{R,P,S hashi}
The curvatures \, $\overdiamond{R}$,  and   \, $\overdiamond{P}$  of  Hashiguchi connection satisfy the  properties:
\begin{description}
  \item[(a)] $\overdiamond{R}(\eta ,\xi )S=\mathfrak{R}(\eta ,\xi ).$
  \item[(b)] $\overdiamond{P}(\eta ,\xi )S=\overdiamond{P}(\eta ,S)\xi =\overdiamond{P}(S,\eta )\xi =0.$
  \item[(c)] $\overdiamond{Q}(\eta ,\xi )S=\overdiamond{Q}(\eta ,S)\xi =\overdiamond{Q}(S,\eta )\xi =0.$
\end{description}
\end{prop}
\begin{proof}
~\par
\noindent \textbf{(a)} Follows from (\ref{c(s)}), Lemma \ref{car.curv.} and Proposition \ref{hashicurv.}.

\noindent \textbf{(b)} By making use of  (\ref{c(s)}),  the facts that \,$\overcirc{P}(\eta ,\xi )S=\overcirc{P}(\eta ,S)\xi
=0$ \cite{Nabil.1} and the bracket $[h\eta ,C]$ is horizontal, we get:
\begin{eqnarray*}
   \overdiamond{P}(\eta ,\xi )S&=&(\,\overdiamond{D}_{h\eta }\mathcal{C})(\xi ,S) \\
   &=&-\mathcal{C}(\xi ,\,\overdiamond{D}_{h\eta }hS) \\
   &=&-\mathcal{C}(\xi ,Fv[h\eta ,C]) \\
   &=&0.
\end{eqnarray*}
And \, $\overdiamond{P}(\eta ,S)\xi =0$ can be proven in a similar manner and hence  $\,\overdiamond{P}(S,\eta )\xi =0$  by the antisymmetric property  of $\,\overdiamond{P}$.

\noindent \textbf{(c)} Follows from (\ref{c(s)}) and Proposition \ref{hashicurv.}.
\end{proof}

To investigate the Bianchi identities for Hashiguchi connection,  the following lemma is required.
\begin{lem}\label{generalbianchi} Consider  a linear  Finsler connection $\textbf{D}$   on $TM$
with the  classical torsion  tensor $\textbf{T}$ and
the classical  curvature tensor $\textbf{K}$. Then, for all
$\zeta ,\eta ,\xi  \in \cppp$, we have the following identities:
\begin{description}
     \item[(a)]$\mathfrak{S}_{\zeta ,\eta ,\xi }\{\textbf{K}(\zeta ,\eta ) \xi \}=\mathfrak{S}_{\zeta ,\eta ,\xi }\{{\textbf{T}(\textbf{T}}(\zeta ,\eta ),\xi )
   +(\textbf{D}_\zeta {T})(\eta ,\xi ) \}$,

\item[(b)]
$\mathfrak{S}_{\zeta ,\eta ,\xi }\{{\textbf{K}}({T}(\zeta ,\eta ),\xi )+(\textbf{D}_\zeta \textbf{K})(\eta ,\xi )\}=0$,
\end{description}
where the notation  $\mathfrak{S}_{\zeta ,\eta ,\xi }$ refers to the  cyclic sum over
$\zeta $,$\eta $ and $\xi $.
\end{lem}

We are now ready to explore the Bianchi identities for the Hashiguchi connection, as detailed in the following proposition.
\begin{prop}\label{bianchihashi}
For for Hashiguchi connection on a Finsler manifold $(M,E)$, then we have the following Bianchi identities:
\begin{description}
  \item[(a)]$\mathfrak{S}_{\zeta ,\eta ,\xi }\{\, \overdiamond{R}(\zeta ,\eta )\xi \}=\mathfrak{S}_{\zeta ,\eta ,\xi }\{\mathcal{C}(F\mathfrak{R}(\zeta ,\eta ),\xi )\}$.

  \item[(b)] $\mathfrak{S}_{\zeta ,\eta ,\xi }\{\, \overdiamond{Q}(\zeta ,\eta )\xi \}=0$.

  \item[(c)] $\mathcal{C}(F\mathfrak{R}(\zeta ,\eta ),\xi )=\mathfrak{R}(F\mathcal{C}(\zeta ,\xi ),\eta )-\mathfrak{R}(F\mathcal{C}(\eta ,\xi ),\zeta )$.

  \item[(d)] $\mathfrak{S}_{\zeta ,\eta ,\xi }\{(\, \overdiamond{D}_{h\zeta }\mathfrak{R})(\eta ,\xi )\}=0$.

  \item[(e)]$\mathfrak{S}_{\zeta ,\eta ,\xi }\,\{(\, \overdiamond{D}_{h\zeta }\, \overdiamond{R})(\eta ,\xi )\}=\mathfrak{S}_{\zeta ,\eta ,\xi }\,\{\, \overdiamond{P}(\zeta ,F\mathfrak{R}(\eta ,\xi ))\}$.

  \item[(f)]$(\, \overdiamond{D}_{h\zeta }\,\overdiamond{P})(\eta ,\xi )-(\, \overdiamond{D}_{h\eta }\, \overdiamond{P})(\zeta ,\xi )+(\, \overdiamond{D}_{J\xi }\, \overdiamond{R})(\zeta ,\eta )=
  \, \overdiamond{R}(F{\mathcal{C}}(\eta ,\xi ),\zeta )\\-\, \overdiamond{R}(F{\mathcal{C}}(\zeta ,\xi ),\eta )-\, \overdiamond{Q}(F\mathfrak{R}(\zeta ,\eta ),\xi )$.

  \item[(g)]$(\, \overdiamond{D}_{h\zeta }\, \overdiamond{Q})(\eta ,\xi )-(\, \overdiamond{D}_{J\eta }\, \overdiamond{P})(\zeta ,\xi )+(\, \overdiamond{D}_{J\xi }\, \overdiamond{P})(\zeta ,\eta )=\, \overdiamond{P}(F{\mathcal{C}}(\zeta ,\eta ),\xi )\\
  -\, \overdiamond{P}(F{\mathcal{C}}(\xi ,\zeta ),\eta )$.

  \item[(h)]$\mathfrak{S}_{\zeta ,\eta ,\xi }\{(\, \overdiamond{D}_{J\zeta }\, \overdiamond{Q})(\eta ,\xi )\}=0$,
\end{description}
\end{prop}
\begin{proof}
~\par
\noindent {\textbf{(a)}} Making use  of   Lemma \ref{generalbianchi} \textbf{(a)} for $h\zeta $, $h\eta $ and $h\xi $, we get
\begin{equation}\label{bianchi1}
    F\mathfrak{S}_{\zeta ,\eta ,\xi }\{\,\overdiamond{R}(\zeta ,\eta ) \xi \}=\mathfrak{S}_{\zeta ,\eta ,\xi }\{(\,\,\overdiamond{D}_{h\zeta }\mathfrak{R})(\eta ,\xi )\}.
\end{equation}
Since $Fv=hF$ and $\,\overdiamond{R}$, $\mathfrak{R}$ are verticals, then by comparing the horizontal parts taking the fact that $F$ is an isomorphism into account,  we have

$$\mathfrak{S}_{\zeta ,\eta ,\xi }\{\,\overdiamond{R}(\zeta ,\eta )\xi \}=0.$$

\noindent {\textbf{(b)}} It follows by employing  Lemma \ref{generalbianchi} \textbf{(a)} on $J\zeta $, $J\eta $ and $J\xi $.

\noindent {\textbf{(c)}} Follows by applying Lemma \ref{generalbianchi} \textbf{(a)} on $h\zeta $, $h\eta $ and $J\xi $ and comparing the vertical parts.

\noindent {\textbf{(d)}} The result follows by
comparing the vertical parts in (\ref{bianchi1}).

\noindent {\textbf{(e)}} By using   Lemma \ref{generalbianchi} \textbf{(b)} on $h\zeta $, $h\eta $ and $h\xi $.

\noindent {\textbf{(f)}} The result comes by implementing   Lemma \ref{generalbianchi} \textbf{(b)} on $h\zeta $, $h\eta $ and $J\xi $.

\noindent {\textbf{(g)}}Follows by considering  Lemma \ref{generalbianchi} \textbf{(b)} on $h\zeta $, $J\eta $ and $J\xi $.

\noindent {\textbf{(h)}} It follows by using  Lemma \ref{generalbianchi} \textbf{(b)} on $J\zeta $, $J\eta $ and $J\xi $.
\end{proof}

\begin{prop} The hv-curvature of the Hashiguchi connection $\,\,\overdiamond{P}$ satisfies that
$$\,\overdiamond{P}(\zeta ,\kappa )\xi =\,\overdiamond{P}(\zeta ,\xi )\kappa .$$
\end{prop}
\begin{proof}
Since \, $\overcirc{P}$ is totally symmetric \cite{Nabil.1} and the symmetric property of $\mathcal{C}$, then the result follows by Proposition \ref{hashicurv.}.
\end{proof}

\begin{lem}\label{p,q.hashi.s}
The  hv-curvature \, $\overdiamond{P}$ and the v-curvature $\,\overdiamond{Q}$ of the Hashiguchi connection satisfy the following properties:
$$(\,\overdiamond{D}_{h\zeta }\,\overdiamond{P})(\eta ,S)=0, \quad\quad (\,\overdiamond{D}_{J\zeta }\,\overdiamond{P})(\eta ,S)=\,\overdiamond{P}(\eta ,\zeta ),$$
$$ (\,\overdiamond{D}_{h\zeta }\,\overdiamond{Q})(\eta ,S)=0, \quad\quad (\,\overdiamond{D}_{J\zeta }\,\overdiamond{Q})(\eta ,S)=\,\overdiamond{Q}(\eta ,\zeta ).$$
\end{lem}
\begin{proof}
By  Theorem \ref{hashi.nab.} and since $[h\zeta ,C]$ is horizontal, then $\,\overdiamond{D}_{h\zeta }hS=Fv[h\zeta ,JS]=Fv[h\zeta ,C]=0$. Hence, by  Proposition \ref{R,P,S hashi}, we have $(\,\overdiamond{D}_{h\zeta }\,\overdiamond{P})(\eta ,S)=(\,\overdiamond{D}_{h\zeta }\,\overdiamond{Q})(\eta ,S)=0$. Making use of (\ref{c(s)}) and Lemma \ref{jx}, we get $\,\overdiamond{D}_{J\zeta }hS=h[J\zeta ,S]+F\mathcal{C}(\zeta ,S)=h\zeta $. Thus, $(\,\overdiamond{D}_{J\zeta }\,\overdiamond{P})(\eta ,S)=\,\overdiamond{P}(\eta ,\zeta )$ and $(\,\overdiamond{D}_{J\zeta }\,\overdiamond{Q})(\eta ,S)=\,\overdiamond{Q}(\eta ,\zeta )$.
\end{proof}
\begin{prop}\label{dhashi.C.RPQ} The curvatures  \, $\overdiamond{R}$,    \, $\overdiamond{P}$,  and   $\,\overdiamond{Q}$ of the Hashiguchi connection satisfy the following facts:
$$\,\,\overdiamond{D}_C\,\,\overdiamond{R}=0, \quad\quad \,\,\overdiamond{D}_C\,\,\overdiamond{P}=-\,\overdiamond{P},\quad\quad \,\,\overdiamond{D}_C\,\,\overdiamond{Q}=-2\,\overdiamond{Q}.$$
\end{prop}
\begin{proof}
The result follows from   Lemma \ref{bianchihashi} \textbf{(d)} and  \textbf{(e)},   by putting $\xi =S$.
\end{proof}

We have the following formulae of Lie brackets based on  the  Hashiguchi connection $\,\overdiamond{D}$.

\begin{prop} For the  Hashiguchi connection $\,\overdiamond{D}$ and for all $\eta ,\xi \in \cppp$, we have the following properties:
\begin{description}
  \item[(a)] $[J\eta ,J\xi ]=J(\,\,\overdiamond{D}_{J\eta }\xi -\,\,\overdiamond{D}_{J\xi }\eta ).$
  \item[(b)] $[h\eta ,J\xi ]=J(\,\,\overdiamond{D}_{h\eta })\xi -h(\,\,\overdiamond{D}_{J\xi }\eta )+F{\mathcal{C}}(\eta ,\xi ).$
  \item[(c)] $[h\eta ,h\xi ]=h(\,\,\overdiamond{D}_{h\eta }\xi -\,\,\overdiamond{D}_{h\xi }\eta )-\mathfrak{R}(\eta ,\xi ).$
\end{description}
\end{prop}

\begin{proof}~\par
\noindent \textbf{(a)} By Theorem \ref{hashi.nab.}, we get
\begin{eqnarray*}
   J(\,\,\overdiamond{D}_{J\eta }\xi -\,\,\overdiamond{D}_{J\xi }\eta )&=&\,\,\overdiamond{D}_{J\eta }J\xi -\,\,\overdiamond{D}_{J\xi }J\eta   \\
   &=&J[J\eta ,\xi ]-J[J\xi ,\eta ]\\
   &=&[J\eta ,J\xi ].
 \end{eqnarray*}
\noindent \textbf{(b)} Since the tensor $\mathcal{C}$ is symmetric, then we obtain
\begin{eqnarray*}
   J(\,\,\overdiamond{D}_{h\eta }\xi )-h(\,\,\overdiamond{D}_{J\xi }\eta )&=&\,\,\overdiamond{D}_{h\eta }J\xi -\,\,\overdiamond{D}_{J\xi }h\eta   \\
   &=&v[h\eta ,J\xi ]-h[J\xi ,\eta ] -F{\mathcal{C}}(\xi ,\eta ) \\
   &=&v[h\eta ,J\xi ]-h[J\xi ,\eta ]-F{\mathcal{C}}(\xi ,\eta ) \\
   &=&v[h\eta ,J\xi ]+h[h\eta ,J\xi ]-F{\mathcal{C}}(\xi ,\eta ) \\
   &=&[h\eta ,J\xi ]-F{\mathcal{C}}(\eta ,\xi ) .
 \end{eqnarray*}
 \noindent \textbf{(c)} Once again, by making use of the symmetry property of the tensor ${\mathcal{C}}$, we can write
\begin{eqnarray*}
   h(\,\overdiamond{D}_{h\eta }\xi -\,\overdiamond{D}_{h\xi }\eta )&=&\,\overdiamond{D}_{h\eta }h\xi -\,\overdiamond{D}_{h\xi }h\eta   \\
      &=&Fv[h\eta ,J\xi ]+Fv[J\eta ,h\xi ].
    \end{eqnarray*}
 Since  the torsion of the connection $\Gamma$  vanishes, then we have $$0=t(\eta ,\xi )=v[J\eta ,h\xi ]+v[h\eta ,J\xi ]-J[h\eta ,h\xi ].$$
 Which implies   $Fv[J\eta ,h\xi ]+Fv[h\eta ,J\xi ]=FJ[h\eta ,h\xi ]=h[h\eta ,h\xi ]$. Consequently, we have
 $$ h(\,\overdiamond{D}_{h\eta }\xi -\,\overdiamond{D}_{h\xi }\eta ) =h[h\eta ,h\xi ]=[h\eta ,h\xi ]-v[h\eta ,h\xi ]=
   [h\eta ,h\xi ]+\mathfrak{R}(\eta ,\xi ).$$
 It should be noted that we have used the fact  $\mathfrak{R}(\eta ,\xi )=-v[h\eta ,h\xi ]$, for example, see \cite{Nabil.2}.
\end{proof}

\section*{Concluding remarks}
Let's end this work by the following comments and remarks:
\begin{itemize}
\item   For a Finsler manifold $(M,E)$, following the  Klein-Grifone approach, we can
 associate   four fundamental linear connections {{canonically}}: the
Berwald  connection  $\overcirc{D}$  (Theorem  \ref{Th:Berwald_connec.}), the Cartan connection
$D$ (Theorem  \ref{Th:Cartan_connec.}), the Chern  connection $\overast{D}$
(Theorem  \ref{Th:Chern_connec.}) and   the  Hashiguchi connection $\,\overdiamond{D}$ (Theorem
\ref{hashiconnc.}). For each of these connections, the underlying nonlinear connection is the Barthel connection.

\item Define the $\C$-process by adding   $\C$ to the vertical counterpart of the connection, and the $\C'$-process by adding  $\C'$ to the horizontal counterpart. Now, we can convert one  Finsler connection   to some other connections, as follows:
 by Theorem
 \ref{hashiconnc.}, we observe that  the Hashiguchi connection $\,\overdiamond{D}$ is obtained
  from the Berwald connection $\overcirc{D}$ by $\C$-process. Moreover,
  by Theorem \ref{Th:Cartan_connec.},    the Cartan
   connection $D$ is obtained from the Hashiguchi connection
   by $\C'$-process. Now, applying the $\C'$-process  on Hashiguchi connection $\,\overdiamond{D}$, the Chern connection $\overast{D}$ is obtained. Then applying $\C$-process on Chern connection $\overast{D}$, the Cartan connection $D$ is obtained. That is, we have

 $$ D  \xlongleftarrow{\text {  $\C$-process }} \, \overast{D}  \xlongleftarrow{\text {  $\C'$-process }}  \, \overcirc{D}   \xlongrightarrow{\text {  $\C$-process }}   \, \overdiamond{D}    \xlongrightarrow{\text {  $\C'$-process }}  D.
$$

\item We offer a comparative analysis of the four fundamental Finsler connections in the KG-approach, which provides an intrinsic perspective on global Finsler geometry. The following table summarizes the key differences between these connections, including their associated canonical linear connections and geometric objects. For a detailed treatment of the Chern connection, refer to \cite{Chern}.
 \end{itemize}
\begin{landscape}
 \begin{center}{\bf{Table 1: The four fundamental Finsler Linear connections }}
\end{center}
\begin{center}
\small{\begin{tabular}
{|c|c|c|c|c|}\hline
&&&&\\
 {\bf Connection} &{\bf  Berwald: \,$\overcirc{D}$ }& {\bf Cartan:\,$D$ } &{\bf Chern:  \,\overast{D}  }
 &{\bf Hashiguchi: \,\overdiamond{D}}
\\[0.1 cm]\hline
&&&&\\
{\bf Expressions  } & $\overcirc{D}_{J\zeta }J\eta =J[J\kappa ,\eta ]$&
 $D_{J\kappa }J\eta =\overcirc{D}_{J\kappa }J\eta +\mathcal{C}(\kappa ,\eta )$&
$\,\overast{D}_{J\kappa }J\eta =J[J\kappa ,\eta ]$&
$\,\overdiamond{D}_{J\kappa }J\eta =J[J\kappa ,\eta ]+{\mathcal{C}}(\kappa ,\eta )$
\\[0.1 cm]
{\bf of connection} & $\overcirc{D}_{h\kappa }J\eta =v[h\kappa ,J\eta ]$&
$D_{h\kappa }J\eta =\overcirc{D}_{h\kappa }J\eta +\mathcal{C}'(\kappa ,\eta )$&
$\,\overast{D}_{h\kappa }J\eta =v[h\kappa ,J\eta ]+\mathcal{C}'(\kappa ,\eta )$&
      $\,\overdiamond{D}_{h\kappa }J\eta =v[h\kappa ,J\eta ]$\\
&&&&\\
& $\overcirc{D}F=0$&${D}F=0$&$\,\overast{D} F=0$&$\,\overdiamond{D} F=0$
\\[0.1 cm]\hline
&&&&\\
{\bf hh-torsions} & $\mathfrak{R}$& $\mathfrak{R}$& $\mathfrak{R}$&$\mathfrak{R}$
\\[0.1 cm]
{\bf} hv-torsions & $0$& $\mathcal{C}'-F\mathcal{C}$& $\mathcal{C}'$&$-F\mathcal{C}$
\\[0.1 cm]
{\bf  vv-torsions} & $0$& $0$&$0$&$0$
\\[0.1 cm]\hline
&&&&\\
{\bf h-curvature} & $\overcirc{R}$& $ R(\kappa ,\eta )\xi =\overcirc{R}(\kappa ,\eta )\xi +(D_{h\kappa }\mathcal{C}')(\eta ,\xi ) $& $\overast{R}(\kappa ,\eta )\xi =R(\kappa ,\eta )\xi $&$\overdiamond{R}(\kappa ,\eta )\xi =\overcirc{R}(\kappa ,\eta )\xi $
\\
&&$-(D_{h\eta }\mathcal{C}')(\kappa ,\xi )
      +\mathcal{C}'(F\mathcal{C}'(\kappa ,\xi ),\eta )$&$-\mathcal{C}(F\mathfrak{R}(\kappa ,\eta ),\xi )$&$+\mathcal{C}(F\mathfrak{R}(\kappa ,\eta ),\xi )$\\
    && $-\mathcal{C}'(F\mathcal{C}'(\eta ,\xi ),\kappa )+\mathcal{C}(F\mathfrak{R}(\kappa ,\eta ),\xi )$&&\\
{\bf hv-curvature} & $\overcirc{P}$& $ P(\kappa ,\eta )\xi  =\overcirc{P}(\kappa ,\eta )\xi +(D_{h\kappa }{\mathcal{C}})(\eta ,\xi )$& $\overast{P}(\kappa ,\eta )\xi =\overcirc{P}(\kappa ,\eta )\xi $&$\overdiamond{P}(\kappa ,\eta )\xi =\overcirc{P}(\kappa ,\eta )\xi $\\
&&$-(D_{J\eta }\mathcal{C}')(\kappa ,\xi )+\mathcal{C}(F\mathcal{C}'(\kappa ,\xi ),\eta ) $&$-(\,\,\overast{D}_{J\eta }\mathcal{C}')(\kappa ,\xi )$&$+(\,\,\overdiamond{D}_{h\kappa }\mathcal{C})(\eta ,\xi )$\\
&&$+{\mathcal{C}}(F\mathcal{C}'(\kappa ,\eta ),\xi )-\mathcal{C}'(F{\mathcal{C}}(\eta ,\xi ),\kappa )$&&\\
&&$-\mathcal{C}'(F{\mathcal{C}}(\kappa ,\eta ),\xi )$&&\\
{\bf v-curvature} & $0$& $ Q(\kappa ,\eta )\xi ={\mathcal{C}}(F{\mathcal{C}}(\kappa ,\xi ),\eta ) $& $0$&$\overdiamond{Q}=Q$\\
&&$-{\mathcal{C}}(F{\mathcal{C}}(\eta ,\xi ),\kappa )$&&\\\hline
{\bf v-metricity}& not v-metrical & v-metrical &
not v-metrical & v-metrical
\\[0.1 cm]
{\bf h-metricity}& not h-metrical&  h-metrical&  h-metrical&not h-metrical
\\[0.1 cm]\hline
\end{tabular}}
\end{center}
\end{landscape}


\section*{ Appendix: Some local formulas}

For completeness, we provide a concise overview of essential geometric objects' local expressions in this appendix. 

  Let   $(M,E)$ be a  Finsler manifold,   $(U,(x^{i}))$ be  a system  of local coordinates on
 $M$,  and $(\pi^{-1}(U),(x^i,y^i))$ the attached system of   coordinates on $TM$.

 We have the following objects:\\

 \noindent $(\pa_{i}):=(\frac{\pa}{\pa x^i})$: the natural bases of $T_{x}M,\, x\in
 M$,\\
 $(\paa_{i}):=(\frac{\pa}{\pa y^i})$: the natural bases of $V_{u}(TM),\, u\in
 TM$,\\
$(\pa_{i},\paa_{i})$: the natural bases of $T_{u}(TM)$,\\
$(\overline{\pa}_{i} )$: the natural bases of the fiber over $u$ in
$\p$.\\

\noindent $g_{ij}:=  \paa_{i} \paa_{j}E$: the components of the  metric tensor, where $E=\frac{1}{2} \textbf{L}^2$ is the energy function, and $\textbf{L}$ is the Finsler structure,
the
Finsler metric tensor,\\
$g^{hr}$:   the components of the inverse metric tensor,\\
$G^h:=\frac{1}{2}g^{hr}(y^r\partial_r\dot{\partial}_iE - \partial_iE)$: the components of the canonical spray,\\
$N^{h}_{i}:=\paa_{i}\,G^h$: the components of the canonical nonlinear connection,\\
$G^{h}_{ij}:=\paa_{j}\,N^h_{i}=\paa_{j}\paa_{i}\,G^h$: the coefficients of Berwald connection,\\
$(\delta_{i}):=(\pa_{i}-N^{h}_{i}\,\paa_{h})$: the adapted bases of
$H_{u}(TM)$,\\
$(\delta_{i}, \paa_{i})$: the bases of $T_{u}(TM)=H_{u}(TM)\oplus
V_{u}(TM)$. \\

\vspace{5pt}

\noindent $\gamma^{h}_{ij}:= \frac{1}{2}\,g^{h\ell}(\pa_{i}\,g_{\ell
j}+\pa_{j}\,g_{i\ell }- \pa_{\ell}\,g_{i j}
 )$,\\
  $C^{h}_{ij}:= \frac{1}{2}\,g^{h\ell}(\paa_{i}\,g_{\ell j}+\paa_{j}\,g_{i\ell }-
   \paa_{\ell}\,g_{i j})=
  \frac{1}{2}\,g^{h\ell}\,\paa_{i}\,g_{\ell j}$,\\
  $\Gamma^{h}_{ij}:= \frac{1}{2}\,g^{h\ell}(\delta_{i}\,g_{\ell j}+\delta_{j}\,g_{i\ell }-
  \delta_{\ell}\,g_{i j})$: the coefficients of Caratn connection.\\

The bundle moriphisms $\gamma$ and $\rho$ are acting on the bases as follows:
$$\gamma(\overline{\pa}_i)=\paa_i, \quad \rho(\pa_i)=\overline{\pa}_i,\quad\rho(\paa_i)=0.$$
 The local formula of the almost tangent structure $J$ is given by $J=\paa_i \otimes dx^i$ and we have
$$ J(\pa_{i})= \paa_{i}, \quad  J(\paa_{i})=0.$$
The horizontal and vertical protections are given, locally, by
$$
h=  dx^{i} \otimes \pa_{i}- N^{i}_{j}\, dx^{j}
\otimes \paa_{i} , \quad  v=dy^{i} \otimes \paa_{i}+
N^{i}_{j}\, dx^{j} \otimes \paa_{i}. $$
Moreover, we have
$$h(\pa_i)=\delta_i,\quad h(\paa_i)=0,\quad v(\pa_i)=0,\quad v(\paa_i)=\paa_i.$$
The local formula of $F$ can be found in  \cite[Eq. (I.70)]{r21}, moreover, we have
$$F(\pa_i)= N^j_i\pa_j-N^h_iN^j_h\paa_j-\paa_i , \quad F(\paa_i)=  \delta_i,\quad F(\delta_i)=- \paa_i .$$

\newpage
\begin{center}
\textbf{Table 2: A comparison between the fundamental connections in KG- and PB-approachs in Finsler geoemtry
}\end{center}
\begin{center}
\begin{tabular}{|l|l|l|}
\hline
\textbf{The object} & \textbf{KG-approach} & \textbf{PB-approach} \\ \hline
     The Universe                 &     $(T\T M,\pi_{\T M},\T M)$        &    $(\pi^{-1}(\T M),P,\T M) $        \\ \hline
         Fibers                   &        $T_u\T M$, $u\in \T M$     &        $\{v\}\times T_xM$, $\pi(v)=x$      \\ \hline
 Berwald connection $\overcirc{D}$  &   $\overcirc{D}_{\paa_j}\paa_i=0$     &   $\overcirc{D}_{{\paa}_j}\overline{\pa}_i=0$         \\
        &$\overcirc{D}_{\delta_j}\paa_i=G^k_{ij}\paa_k $  &  $\overcirc{D}_{\delta_j}\overline{\pa}_i=G^k_{ij}\overline{\pa}_k $  \\
         & $\overcirc{D}_{\paa_j}\delta_i=0$   & \\
     &$\overcirc{D}_{\delta_j}\delta_i=G^k_{ij}\delta_k$& \\ \hline
     Cartan connection   $ {D}$ &   $ {D}_{\paa_j}\paa_i= {C}^k_{ij}\paa_k $     &   $ {D}_{{\paa}_j}\overline{\pa}_i= {C}^k_{ij}\overline{\pa}_k $         \\
        &$ {D}_{\delta_j}\paa_i=\Gamma^k_{ij}\paa_k $  &  $ {D}_{\delta_j}\overline{\pa}_i=\Gamma^k_{ij}\overline{\pa}_k $  \\
         & $ {D}_{\paa_j}\delta_i= {C}^k_{ij}\delta_k $   & \\
     &$ {D}_{\delta_j}\delta_i=\Gamma^k_{ij}\delta_k$& \\ \hline
     Chern  connection $ \overast{D}$  &   $ \overast{D}_{\paa_j}\paa_i=0$     &   $ \overast{D}_{{\paa}_j}\overline{\pa}_i =0$         \\
        &$ \overast{D}_{\delta_j}\paa_i=\Gamma^k_{ij}\paa_k $  &  $ \overast{D}_{\delta_j}\overline{\pa}_i=\Gamma^k_{ij}\overline{\pa}_k $  \\
       & $ \overast{D}_{\paa_j}\delta_i= 0 $   & \\
     &$ \overast{D}_{\delta_j}\delta_i=\Gamma^k_{ij}\delta_k$& \\ \hline
      Hashiguchi connection  $\overdiamond{D}$ &   $\overdiamond{D}_{\paa_j}\paa_i={C}^k_{ij}\paa_k$     &   $\overdiamond{D}_{{\paa}_j}\overline{\pa}_i={C}^k_{ij}\overline{\pa}_k$         \\
        &$\overdiamond{D}_{\delta_j}\paa_i=G^k_{ij}\paa_k $  &  $\overdiamond{D}_{\delta_j}\overline{\pa}_i=G^k_{ij}\overline{\pa}_k $  \\
         & $\overdiamond{D}_{\paa_j}\delta_i={C}^k_{ij}\delta_k$   & \\
     &$\overdiamond{D}_{\delta_j}\delta_i=G^k_{ij}\delta_k$& \\ \hline
\end{tabular}
\end{center}

\subsection*{Local formulas of the curvature tensors}
The curvature tensors associated with the fundamental connections on the double tangent bundle annihilate vertical vectors. Consequently, their non-trivial components are confined to the horizontal bundle, as illustrated by the Berwald h-curvature:
$$\overcirc{R}(\delta_i,\paa_j)\delta_k=0, \quad \overcirc{R}(\delta_i,\paa_j)\paa_k=0, \quad \overcirc{R}(\delta_i,\delta_j)\delta_k={R^{\circ}}^{h}_{ijk}\paa_h.$$
In what follow, we list the formulas of the components of the torsion and curvature tensor sof the four fundamental tensors on the double tangent bundle.\\

For the Berwald connection, we have\,: \\
$(v)h$-torsion\,:
${R^{\circ}}^{i}_{jk}=\delta_{k}G^{i}_{j}-\delta_{j}G^{i}_{k}$,\\
$h$-curvature\,:  ${R^{\circ}}^{i}_{hjk}=\mathfrak{A}_{jk}\{
\delta_{k}G^{i}_{hj}+G^{m}_{hj} G^{i}_{mk}\}$, where $\mathfrak{A}_{jk}X_{jk}=X_{jk}-X_{kj}$,\\
$v$-curvature\,:
${S^{\circ}}^{i}_{hjk}\equiv 0$,\\
$hv$-curvature\,:
${P^{\circ}}^{i}_{hjk}=\paa_{k}G^{i}_{hj}=:G^{i}_{hjk}$.

\vspace{7pt}
\par
For the Cartan connection, we have\,: \\
$(v)h$-torsion\,:
$R^{i}_{jk}=\delta_{k}G^{i}_{j}-\delta_{j}G^{i}_{k}={R^{\circ}}^{i}_{jk}$,\\
$(v)hv$-torsion\,: $P^{i}_{jk}=
G^{i}_{jk}-\Gamma^{i}_{jk}$,\\
$(h)hv$-torsion\,: $C^{i}_{jk}=\frac{1}{2} g^{ri}\paa_{r}g_{jk}$,\\
$h$-curvature\,:  $R^{i}_{hjk}=\mathfrak{A}_{jk}\{
\delta_{k}\Gamma^{i}_{hj}+\Gamma^{m}_{hj}
\Gamma^{i}_{mk}\}-C^{i}_{hm}
R^{m}_{jk}$,\\
$hv$-curvature\,:
$P^{i}_{hjk}=\paa_{k}\Gamma^{i}_{hj}-C^{i}_{hk|j}+C^{i}_{hm}P^{m}_{jk}$, where the symbol "$|$" refers to the horizontal covarient derivative with regard the Cartan connection,\\
$v$-curvature\,:
$S^{i}_{hjk}=C^{m}_{hk}C^{i}_{mj}-C^{m}_{hj}C^{i}_{mk}=\mathfrak{A}_{jk}\{
C^{m}_{hk}C^{i}_{mj}\}$.

 \vspace{5pt}
\par

For the Chern connection, we have\,: \\
$(v)h$-torsion\,: $\,R^{\ast
i}_{\,\,\,jk}=\delta_{k}\,G^{i}_{j}-\delta_{j}\,G^{i}_{k}=R^{i}_{jk}$,\\
$(v)hv$-torsion\,: $\,P^{\ast i}_{\,\,\,jk}=
G^{i}_{jk}-\Gamma^{i}_{jk}$,\\
$h$-curvature\,: $\,R^{\ast i}_{\,\,\,hjk}=\mathfrak{A}_{jk}\{
\delta_{k}\,\Gamma^{i}_{hj}+\Gamma^{m}_{hj} \,\Gamma^{i}_{mk}\}$,\\
$hv$-curvature\,: $\,P^{\ast
i}_{\,\,\,hjk}=\paa_{k}\,\Gamma^{i}_{hj}$.

 \vspace{7pt}
\par
For the Hashiguchi connection, we have \\
$(v)h$-torsion $R^{\star i}_{\,\,\,jk}=\delta_{k}\,G^{i}_{j}-\delta_{j}\,G^{i}_{k}=R^{i}_{jk}$,\\
$(h)hv$-torsion $C^{\star i}_{\,\,\,jk}={1}/{2}\{ g^{ri}\,\paa_{r}\,g_{jk}\}=C^{i}_{jk}$,\\
$h$-curvature  $R^{\star i}_{\,\,\,hjk}={R^{\circ}}^{i}_{hjk}+C^{i}_{hm}\,
R^{m}_{jk}$,\\
$hv$-curvature
$P^{\star i}_{\,\,hjk}=\paa_{k}\,G^{i}_{hj}-C^{i}_{hk\stackrel{\star}|j}$, where the symbol "$\stackrel{\star}|$" refers to the horizontal covarient derivative with regard the Hashiguchi connection,\\\\
$v$-curvature  $S^{\star i}_{\,\,\,hjk}=\mathfrak{A}_{jk}\{
C^{m}_{hk}\,C^{i}_{mj}\}=S^{i}_{hjk}$.

\providecommand{\bysame}{\leavevmode\hbox
to3em{\hrulefill}\thinspace}
\providecommand{\MR}{\relax\ifhmode\unskip\space\fi MR }
\providecommand{\MRhref}[2]{%
  \href{http://www.ams.org/mathscinet-getitem?mr=#1}{#2}
} \providecommand{\href}[2]{#2}

\end{document}